\documentclass[11pt,reqno,a4paper]{amsart}
\usepackage{euscript}

\usepackage{xcolor}

\newtheorem{theorem}{Theorem}

\newtheorem{lemma}{Lemma}

\newtheorem{remark}{Remark}

\renewcommand{\epsilon}{\varepsilon}
\DeclareMathOperator{\Ima}{Im}
\DeclareMathOperator{\Ker}{Ker}
\DeclareMathOperator{\esssup}{esssup}
\def\Id{\text{\rm Id}}
\def\cA{\EuScript{A}}

\def\N{\mathbb{N}}
\def\Z{\mathbb{Z}}
\def\R{\mathbb{R}}

\begin{document}

\title[On the linearization of random dynamics] {On the linearization of infinite-dimensional random dynamical systems}

\author{Lucas Backes}
\address{\noindent Departamento de Matem\'atica, Universidade Federal do Rio Grande do Sul, Av. Bento Gon\c{c}alves 9500, CEP 91509-900, Porto Alegre, RS, Brazil.}
\email{lucas.backes@ufrgs.br}

\author{Davor Dragi\v cevi\' c}
\address{Department of Mathematics, University of Rijeka, Croatia}
\email[Davor Dragi\v cevi\' c]{ddragicevic@math.uniri.hr}


\maketitle

  \begin{center}
\textit{Dedicated to the memory of Vlado Dragi\v cevi\' c }
\end{center}
  
\begin{abstract}
We present a new version of the Grobman-Hartman's linearization theorem for random dynamics. Our result holds for infinite dimensional systems whose linear part is not necessarily invertible. In addition, by adding some restrictions on the non-linear perturbations, we don't require for the linear part to be nonuniformly hyperbolic in the sense of Pesin but rather (besides requiring the existence of stable and unstable directions) allow for the existence of a third (central) direction on which we don't prescribe any behaviour for the dynamics. Moreover, under some additional nonuniform growth condition, we prove that the conjugacies given by the linearization procedure are H\"older continuous when restricted to bounded subsets of the space.

\end{abstract}

\section{Introduction}
The classical Grobman-Hartman theorem represents one of the cornerstones of the modern dynamical systems theory. It asserts that if $x_0$ is a hyperbolic fixed point of a $C^1$-diffeomorphism $f\colon \R^n  \to \R^n$, then $f$ is on a neighborhood of $x_0$ topologically conjugated to $Df(x_0)$.
The original references for this results are the papers of Grobman~\cite{G} and Hartman~\cite{H}. This result was extended to the case of Banach spaces independently by Palis~\cite{Palis} and Pugh~\cite{Pugh}, who also  simplified the original arguments of Grobman and
Hartman. It is well-known that even if $f$ is $C^\infty$, the conjugacy can fail to be even locally Lipschitz. Indeed, it turns out that in general the conjugacy is only  locally H\"older continuous~\cite{SX}. The first versions of the Grobman-Hartman theorem for nonautonomous dynamics were established by 
Palmer~\cite{Palmer}  in the case of continuous time  and by Aulbach and Wanner~\cite{AW} in the case of discrete time. 

In the context of random dynamical systems, the first version of the Grobman-Hartman theorem was proved by Wanner~\cite{Wanner}. It represents an important contribution to the study of the qualitative behaviour of random dynamical systems  which are used as models of many phenomena with some degree of noise
(including stochastic differential equations). For a detailed exposition of the theory of random dynamical systems we refer to~\cite{Arnold}. For certain extensions of the original approach developed by Wanner, we refer to~\cite{BDV, CMR, CR, ZS} and references therein. 
In~\cite{BV}, Barreira and Valls formulated sufficient conditions under which the conjugacies in the random version of the Grobman-Hartman theorem are locally H\"older continuous. The higher regularity of conjugacies was discussed in the important works of Li and Lu~\cite{LL1} and 
Lu, Zhang and Zhang~\cite{LZZ}.

The goal of the present paper is to make further contributions to the linearization theory of random dynamical systems. In sharp contrast to the existing results in the literature, our first main result (see Theorem~\ref{theo: hg}) gives a random version of the Grobman-Hartman theorem such that the linear part of our dynamics
is not necessarily invertible (although the associated perturbed dynamics needs to be invertible). In addition, we don't require for the linear part to be nonuniformly hyperbolic in the sense of Pesin but rather (besides requiring the existence of stable and unstable directions) allow for the existence of a third (central) direction on which we don't prescribe any behaviour for the linear part.
On the other hand, we stress that  our result doesn't allow any type of perturbations along this central direction. Consequently, to the best of our knowledge, our Theorem~\ref{theo: hg} represents the first version of the Grobman-Hartman theorem in the setting when the linear part of the dynamics can have zero Lyapunov exponents (assuming that a suitable version of the multiplicative ergodic theorem can be applied to it).

In our second main result (see Theorem~\ref{HL}),  we formulate sufficient conditions  under which the conjugacies given by Theorem~\ref{theo: hg} are H\"older continuous. When compared with the main result in~\cite{BV}, our Theorem~\ref{HL} is more general and sharper since for instance, we deal with the infinite-dimensional dynamics where a multiplicative ergodic theorem is not always available and, moreover, we don't require for the linear part to be hyperbolic.

Our techniques build on our recent work with K. Palmer~\cite{BDP} (see also~\cite{BD1}), where we obtained similar results in the context of nonautonomous dynamics. However, we stress that the results in the present paper require a nontrivial adjustments of the ideas from~\cite{BDP} that now need to be combined with 
certain ergodic theory tools.

\section{Linearization of random dynamical systems}
Let $X=(X, |\cdot |)$ be an arbitrary Banach space and let $\mathcal B(X)$ denote the space of all bounded linear operators on $X$. The norm on $\mathcal B(X)$ will be denoted by $\lVert \cdot \rVert$.
 Furthermore, let $(\Omega, \mathcal F, \mathbb P)$ be a probability space and let $\sigma \colon \Omega \to \Omega$ be an invertible $\mathbb P$-preserving measurable transformation. For the sake of simplicity of the presentation we will assume that $\mathbb P$ is ergodic (see Remark \ref{rem: non ergodic} for the discussion regarding the general case).

We recall that a measurable map $K\colon \Omega \to (0, \infty)$ is said to be \emph{tempered} if:
\begin{equation}\label{eq: def tempered}
\lim_{n\to \pm \infty}\frac 1 n \log K(\sigma^n (\omega))=0, \quad \text{for $\mathbb P$-a.e. $\omega \in \Omega$.}
\end{equation}
Let $A\colon \Omega \to \mathcal B(X)$ be a strongly measurable map. We recall that this means that the map $\omega \mapsto A(\omega)x$ is measurable for each $x\in X$. Let $\cA \colon \Omega \times \{0,1, 2, \ldots\}\to \mathcal B(X)$ be the \emph{linear cocycle} generated by $A$. We recall that $\cA$ is given by
\[
\cA(\omega, n)=\begin{cases}
\Id & \text{if $n=0$;}\\
A(\sigma^{n-1}(\omega))\cdots A(\omega) & \text{if $n>0$,}
\end{cases}
\]
for $\omega \in \Omega$ and $n\in \{0, 1, \ldots \}$. Here, $\Id$ denotes the identity operator on $X$. 

We now introduce some assumptions related to $\cA$. Namely, we will assume the following:
\begin{itemize}
\item there exist three families of projections $(\Pi^i(\omega))_{\omega \in \Omega}$, $i\in \{1, 2, 3\}$ on $X$ such that 
\[
\Id=\Pi^1(\omega)+\Pi^2(\omega)+\Pi^3(\omega), \quad \text{for $\mathbb P$-a.e. $\omega \in \Omega$.}
\]
In addition, for each $i\in \{1, 2, 3\}$, the map $\omega \to \Pi^i (\omega)$ is strongly measurable;
\item for $i\in \{1, 2, 3\}$ and $\mathbb P$-a.e. $\omega \in \Omega$,
\begin{equation}\label{eq: invariance projections}
A(\omega)\Pi^i(\omega)=\Pi^i (\sigma (\omega))A(\omega);
\end{equation}
\item for $i\in \{2,3\}$ and for $\mathbb P$-a.e. $\omega \in \Omega$,
\begin{equation}\label{inv}
A(\omega)\rvert_{\Ima \Pi^i (\omega)} \colon \Ima \Pi^i (\omega)\to \Ima \Pi^i (\sigma (\omega)) \quad \text{is invertible;}
\end{equation}
\item there exist $\lambda>0$ and a tempered random variable $K\colon \Omega \to [1, \infty)$ such that for $\mathbb P$-a.e. $\omega \in \Omega$ and $n\ge 0$
\begin{equation}\label{d1}
\lVert \cA(\omega, n)\Pi^1(\omega)\rVert \le K(\omega)e^{-\lambda n}
\end{equation}
and 
\begin{equation}\label{d2}
\lVert \cA(\omega, -n)\Pi^3(\omega)\rVert \le K(\omega)e^{-\lambda n},
\end{equation}
where 
\[
\cA(\omega, -n):=\bigg{(}\cA(\sigma^{-n}(\omega), n)\rvert_{\Ima \Pi^3(\sigma^{-n}(\omega))} \bigg{)}^{-1} \colon \Ima \Pi^3(\omega) \to \Ima  \Pi^3(\sigma^{-n}(\omega)).
\]
\end{itemize}

\begin{remark}
We would like to point out that the above assumptions are natural in the context of the ergodic theory.  Let $X=\R^d$ and assume that 
\[
\int_\Omega \log^+ \|A(\omega) \|\, d\mathbb P(\omega) <+\infty. 
\]
It follows from the version of the multplicative ergodic theorem (MET) established in~\cite{FLQ} that there exist numbers 
\[
-\infty\le \lambda_1 <\lambda_2 <\ldots <\lambda_k 
\]
and for $\mathbb P$-a.e. $\omega \in \Omega$ a measurable decomposition 
\[
X=E_1(\omega) \oplus E_2(\omega) \oplus \ldots \oplus E_k(\omega)
\]
such that 
\[
A(\omega)E_i (\omega) \subset E_i(\sigma (\omega)) \text{(with equality if $\lambda_i >-\infty$)}
\]
and 
\[
\lim_{n\to \infty} \frac 1 n \log |\cA(\omega, n)v |=\lambda_i, \quad \text{for $v\in E_i(\omega)$, $v\neq 0$ and $i\in \{1, \ldots ,k\}$.}
\]
Set 
\[
E^s(\omega):=\bigoplus_{i; \lambda_i <0} E_i(\omega), \quad E^u:=\bigoplus_{i; \lambda_i >0} E_i(\omega),
\]
and in the case when $\lambda_j=0$ for some $j$, let $E^c(\omega):=E_j (\omega)$ (if such $j$ doesn't exist, set $E^c(\omega)=\{0\}$). 
\begin{equation}\label{DEC}
E^s(\omega) \oplus E^c(\omega) \oplus E^u(\omega)= X, \quad \text{for $\mathbb P$-a.e. $\omega \in \Omega$.}
\end{equation}
Let $\Pi^1(\omega) \colon X\to E^s(\omega)$, $\Pi^2(\omega) \colon X\to E^c(\omega)$ and $\Pi^3(\omega) \colon X\to E^u(\omega)$ be projections associated to~\eqref{DEC}. Clearly, \eqref{eq: invariance projections} and~\eqref{inv} holds. Finally, it follows from~\cite[Theorem 2]{DF}
 that there exist $\lambda >0$ and a tempered random variable $K\colon \Omega \to [1, +\infty)$ such that~\eqref{d1} and~\eqref{d2} holds. 

We remark that under some additional assumptions an analogous statement can be formulated in the infinite-dimensional setting (by using~\cite[Proposition 3.7.]{BD})  relying on the infinite-dimensional versions of MET (see~\cite{B, GTQ, LL} and references therein).
\end{remark}

Furthermore, we assume that $(f_\omega)_{\omega \in \Omega}$ and $(g_\omega)_{\omega \in \Omega}$ are two families of continuous maps $f_\omega \colon X\to X$, $g_\omega \colon X\to X$ such that:
\begin{itemize}
\item for every $x\in X$ and $\mathbb P$-a.e. $\omega \in \Omega$, $f_\omega(x), g_\omega (x)\in E^{s,u}(\omega):=\Ker \Pi^2(\omega)$;
\item for every $x\in X$, $\omega \mapsto f_\omega(x)$ and $\omega \mapsto g_\omega(x)$ are measurable;
\item there exists a tempered measurable map $D\colon \Omega \to [1, \infty)$ such that  
\begin{equation}\label{eq: f and g bounded}
\max \{||f_\omega||_\infty, ||g_\omega||_\infty\}\le D(\omega),
\end{equation}
where for any $h\colon X\to X$, we set
\[
\lVert h\rVert_\infty:=\sup_{x\in X} |h(x)|;
\]

\item there exist $c>0$ such that for $\mathbb P$-a.e. $\omega \in \Omega$ and $x, y\in X$,
\begin{equation}\label{l1}
|f_\omega(x)-f_\omega(y)|\le \frac{c}{K(\sigma (\omega))}|x-y|
\end{equation}
and
\begin{equation}\label{l2}
|g_\omega(x)-g_\omega(y)|\le \frac{c}{K(\sigma (\omega))}|x-y|;
\end{equation}
\item for $\mathbb P$-a.e. $\omega \in \Omega$ and $y\in E^c(\omega):=\Ima \Pi^2(\omega)$, 
\begin{equation}\label{eq: A+f invert}
x\mapsto A(\omega)x+f_\omega(x+y)
\end{equation}
and 
\[
x\mapsto A(\omega)x+g_\omega(x+y)
\]
are invertible maps from $E^{s,u}(\omega)$ onto $E^{s,u}(\sigma (\omega))$.
\end{itemize}
\begin{remark}
We observe that if $A(\omega)$ is invertible and $\|A(\omega)^{-1}\| \cdot \text{Lip}(f_\omega)<1$ where $\text{Lip}(f_\omega)$ stands for the Lipschitz constant of $f_\omega$, then $A(\omega)+f_\omega$ is a homeomorphism. Consequently, if $A(\omega)$ is invertible for $\mathbb{P}$-a.e. $\omega\in \Omega$, $a=\esssup_{\omega\in \Omega}\|A(\omega)^{-1}\|<+\infty$ and $ac < 1$, then $A(\omega)+f_\omega$ is a  homeomorphism for $\mathbb{P}$-a.e $\omega\in \Omega$ and, in particular, the last hypothesis above is satisfied (see \cite[p. 433]{BD19}).
\end{remark}

\begin{remark}
\begin{enumerate}
\item In the case when $K$ is a constant random variable, we note that the requirements~\eqref{l1} and~\eqref{l2} mean that $f_\omega$ and $g_\omega$ are Lipschitz with a uniform (independent on $\omega$) Lipschitz constant. We stress that even in this setting, our Theorems~\ref{theo: hg} and~\ref{HL} are new results.

\item Let us now describe a procedure for the construction of maps $f_\omega$ satisfying~\eqref{l1} in the general case. The discussion is essentially taken from~\cite[Section 5.2]{BD2} but we include for the sake of completeness.
 Take an arbitrary $\rho >0$. By~\cite[Proposition 4.3.3]{Arnold}, there exists a random variable $D\colon \Omega \to (0, +\infty)$ such that 
\begin{equation}\label{1706}
K(\omega)\le D(\omega) \quad \text{and} \quad D(\sigma^n (\omega))\le D(\omega)e^{\rho |n|},
\end{equation}
for $\mathbb P$-a.e. $\omega \in \Omega$ and $n\in \Z$. Choose $T>0$ sufficiently large so that $\mathbb P(A)>0$, where
\[
A:=\{ \omega \in \Omega: D(\omega)\le T\}.
\]
For $n\in \N$, set 
\[
A_n:=\sigma^{-n}(A)\setminus \cup_{k=0}^{n-1}\sigma^{-k}(A).
\]
Letting $A_0=A$, we observe that $A_n\cap A_m=\emptyset$ for $n\neq m$. Moreover, due to the ergodicity of $(\Omega, \mathcal F, \mathbb P, \sigma)$, we have that $\mathbb P(\cup_{n=0}^\infty A_n)=1$. For $n\ge 0$ and $\omega \in A_n$, let $f_\omega$ be such that
\begin{equation}\label{1718}
|f_\omega(x)-f_\omega (y)| \le \frac{c}{T}e^{-\rho |n-1|}|x-y|, \quad \text{for $x, y\in X$.}
\end{equation}
We now claim that~\eqref{l1} holds. Take $\omega \in A_n$ for $n\ge 1$. Choose $\omega'\in A$ such that $\omega=\sigma^{-n} (\omega')$. By~\eqref{1706}, we have that
\[
K(\sigma (\omega)) \le D(\sigma (\omega)) =D(\sigma^{-(n-1)} (\omega'))\le e^{\rho |n-1|}D(\omega')\le Te^{\rho |n-1|},
\]
and thus 
\[
\frac c Te^{-\rho |n-1|} \le \frac{c}{K(\sigma (\omega))}.
\]
We conclude that~\eqref{l1} follows from~\eqref{1718}. One can argue in the same manner for $\omega \in A$.
\end{enumerate}
\end{remark}

\begin{remark}
We note that the condition in the spirit of~\eqref{l1} appeared in some of the earlier versions of the Grobman-Hartman theorem for random dynamical systems (see for example~\cite[eq.(21)]{BDV}). Furthermore, it is analogous to the condition (see~\cite[eq.(2.2)]{ZLZ}) imposed in the roughness theorem for tempered exponential dichotomies established in~\cite{ZLZ}.
\end{remark}

We are now in a position to formulate our first result.
\begin{theorem}\label{theo: hg}
For $c>0$ sufficiently small, there exists a family of homeomorphisms $(H_\omega)_{\omega \in \Omega}$, $H_\omega \colon X\to X$ such that:
\begin{itemize}
\item for $\mathbb P$-a.e. $\omega \in \Omega$,
\begin{equation}\label{1lin}
H_{\sigma(\omega)}\circ (A(\omega)+f_\omega)=(A(\omega)+g_\omega)\circ H_\omega;
\end{equation}
\item for $x\in X$, $\omega \mapsto H_\omega(x)$ is measurable;
\item there exists a tempered random variable $T\colon \Omega \to (0, \infty)$ such that
\[
\lVert H_\omega-\Id\rVert_\infty \le T(\omega), \quad \text{for $\mathbb P$-a.e. $\omega \in \Omega$;}
\]
\item $H_{\omega}(x)-x \in E^{s,u}(\omega) $ for $\mathbb P$-a.e. $\omega \in \Omega$ and every $x\in X$.
\end{itemize}
Moreover, the family $(H_\omega)_{\omega \in \Omega}$ satisfying the previous properties is essentially unique.  More precisely, if $(H_\omega')_{\omega \in \Omega}$ is another family that satisfies these properties, then $H_\omega=H_\omega'$ for $\mathbb P$-a.e. $\omega \in \Omega$.
\end{theorem}

\begin{proof}
For $\omega \in \Omega$, we consider two systems:
\begin{equation}\label{S1}
\begin{split}
x(n+1) &=A(\sigma^n (\omega))x(n)+f_{\sigma^n (\omega)}(x(n)+y(n)) \\
y(n+1)&=A(\sigma^n (\omega))y(n),
\end{split}
\end{equation}
and
\begin{equation}\label{S2}
\begin{split}
x(n+1) &=A(\sigma^n (\omega))x(n)+g_{\sigma^n (\omega)}(x(n)+y(n)) \\
y(n+1)&=A(\sigma^n (\omega))y(n).
\end{split}
\end{equation}
For $(\xi, \eta)\in E^{s,u}(\omega)\times E^c(\omega)$, let $n \mapsto (x_1(n, \omega, \xi, \eta), y(n, \omega, \eta))$ denote the solution of~\eqref{S1} such that $(x_1(0, \omega, \xi, \eta), y(0, \omega, \eta))=(\xi, \eta)$. Similarly, let  $n \mapsto (x_2(n, \omega, \xi, \eta), y(n, \omega, \eta))$ denote the solution of~\eqref{S2} satisfying $(x_2(0, \omega, \xi, \eta), y(0, \omega, \eta))=(\xi, \eta)$.

It follows from our assumptions and~\cite[Proposition 4.3.3]{Arnold} that there exists a tempered random variable $C\colon \Omega \to (0, \infty)$ such that for $\mathbb P$-a.e. $\omega \in \Omega$ and $n\in \Z$,
\begin{equation}\label{KDC}
K(\sigma(\omega))D(\omega)\le C(\omega) \quad \text{and} \quad C(\sigma^n (\omega))\le C(\omega)e^{\frac{\lambda}{2}|n|}.
\end{equation}
Let $Y_\infty$ denote the space of all maps $h\colon \Omega \times X\to X$ with the property that:
\begin{itemize}
\item for $\mathbb P$-a.e. $\omega \in \Omega$, $h(\omega, \cdot) \colon X\to X$ is continuous;
\item for $\mathbb P$-a.e. $\omega \in \Omega$ and $x\in X$, $h(\omega, x)\in E^{s,u}(\omega)$;
\item for $x\in X$, $\omega \mapsto h(\omega, x)$ is measurable;
\item \[
||h||_\infty':=\esssup_{\omega \in \Omega}(C(\omega)^{-1}||h(\omega, \cdot)||_\infty)<\infty.
\]
\end{itemize}
It is straightforward to verify that  $(Y_\infty, ||\cdot ||_\infty')$ is a Banach space.  For $\omega \in \Omega$ and $n\in \Z$, set
\[
\mathcal G(\omega, n):=\begin{cases}
\cA(\omega, n)\Pi^1(\omega) & \text{for $n\ge 0$;}\\
-\cA(\omega,n)\Pi^3(\omega) & \text{for $n<0$.}
\end{cases}
\]
For $h\in Y_\infty$, let
\[
(\mathcal T h)(\omega, x):=\sum_{n\in \Z}\mathcal G(\sigma^{-n} (\omega), n)p(n, \omega, x),
\]
where
\begin{equation}\label{eq: def p}
\begin{split}
p(n, \omega, x) &=g_{\sigma^{-(n+1)}(\omega)}(x_1(-(n+1), \omega, \xi, \eta)+y(-(n+1), \omega, \eta) \\
&\phantom{=}+h(\sigma^{-(n+1)}(\omega), x_1(-(n+1), \omega, \xi, \eta)+y(-(n+1), \omega, \eta))) \\
&\phantom{=}-f_{\sigma^{-(n+1)}(\omega)}(x_1(-(n+1), \omega, \xi, \eta)+y(-(n+1), \omega, \eta)),
\end{split}
\end{equation}
for $n\in \Z$, $\omega \in \Omega$ and $x\in X$, where $\xi=(\Id-\Pi^2(\omega))x$, $\eta=x-\xi$.

By~\eqref{d1}, \eqref{d2},\eqref{eq: f and g bounded} and~\eqref{KDC}, we have that 
\begin{equation}\label{ESTIMATES}
\begin{split}
|(\mathcal T h)(\omega, x)| &\le \sum_{n\in \Z}\lVert \mathcal G(\sigma^{-n} (\omega), n)\rVert \cdot (\lVert g_{\sigma^{-(n+1)}(\omega)}\rVert_\infty+\lVert f_{\sigma^{-(n+1)}(\omega)}\rVert_\infty)\\
&\le 2\sum_{n=0}^\infty K(\sigma^{-n}(\omega))e^{-\lambda n}D(\sigma^{-(n+1)}(\omega))\\
&\phantom{\le}+2\sum_{n=1}^\infty K(\sigma^n (\omega))e^{-\lambda n}D(\sigma^{n-1}(\omega))\\
&\le 2\sum_{n=0}^\infty e^{-\lambda n}C(\sigma^{-(n+1)}(\omega))+2\sum_{n=1}^\infty e^{-\lambda n} C(\sigma^{n-1}(\omega))\\
&\le 2e^{\frac{\lambda}{2}}C(\omega)\bigg (\sum_{n=0}^\infty e^{-\frac{\lambda}{2}n}+\sum_{n=1}^\infty e^{-\frac{\lambda}{2}n}\bigg )\\
&\le 2e^{\frac{\lambda}{2}} \frac{1+e^{-\frac{\lambda}{2}}}{1-e^{-\frac{\lambda}{2}}}C(\omega), 
\end{split}
\end{equation}
for $\mathbb P$-a.e. $\omega \in \Omega$ and $x\in X$. Hence, $||\mathcal Th||_\infty' <\infty$.

We now claim that for $\mathbb P$-a.e. $\omega \in \Omega$, $x\mapsto \mathcal Th (\omega, x)$ is a continuous map. We begin by observing that due to the continuity of maps $f_\omega$,  we have that 
\begin{equation}\label{cont}(\xi, \eta)\mapsto (x_1(n, \omega, \xi, \eta), y(n,\omega, \eta)) \quad \text{is continuous,} \end{equation} 
for $\mathbb P$-a.e. $\omega \in \Omega$ and $n\in \Z$.
Fix now an arbitrary $\omega \in \Omega$ for which~\eqref{d1}, \eqref{d2}, \eqref{KDC}  and~\eqref{cont} hold, and such that $h(\sigma^n \omega, \cdot)$ is continuous for each $n\in \Z$. We note that the set of $\omega$'s satisfying these properties is a full measure set.
Take $x\in X$ and $\epsilon >0$. Choose $N>0$ such that 
\[
\sum_{|n|>N}\lVert \mathcal G(\sigma^{-n} (\omega), n)\rVert \cdot (\lVert g_{\sigma^{-(n+1)}(\omega)}\rVert_\infty+\lVert f_{\sigma^{-(n+1)}(\omega)}\rVert_\infty)<\frac{\epsilon}{3}.
\]
We note that the existence of $N$ follows from the estimates in~\eqref{ESTIMATES}. Furthemore, by using~\eqref{cont} together with the continuity of $f_{\sigma^n (\omega)}$, $g_{\sigma^n (\omega)}$ and $h(\sigma^n(\omega), \cdot)$, we conclude that 
\[
\sum_{|n|\le N}\|\mathcal G(\sigma^{-n} (\omega), n)\rVert \cdot |p(n,\omega, x)-p(n,\omega, x')|<\frac{\epsilon}{3},
\]
whenever $|x-x'|$ is sufficiently small. The last two estimates easily imply that 
\[
|(\mathcal T h)(\omega, x)-(\mathcal T h)(\omega, x')| < \epsilon,
\]
whenever $|x-x'|$ is sufficiently small. Hence, $x\mapsto \mathcal Th (\omega, x)$ is a continuous map and 
thus $\mathcal T h\in Y_\infty$. 

Moreover, for $h_1, h_2\in Y_\infty$, we have using~\eqref{l2} that 
\[
\begin{split}
& |(\mathcal Th_1)(\omega, x)-(\mathcal Th_2)(\omega, x)| \\
&\le \sum_{n\in \Z} \frac{c}{K(\sigma^{-n}(\omega))}\lVert  \mathcal G(\sigma^{-n} (\omega), n)\rVert \cdot \lVert h_1(\sigma^{-(n+1)}(\omega), \cdot)-h_2(\sigma^{-(n+1)}(\omega), \cdot)\rVert_\infty \\
&\le c \sum_{n=0}^\infty e^{-\lambda n}C(\sigma^{-(n+1)}(\omega)) \lVert h_1-h_2\rVert_\infty' +c\sum_{n=1}^\infty e^{-\lambda n}C(\sigma^{n-1}(\omega))\lVert h_1-h_2\rVert_\infty' \\
&\le ce^{\frac{\lambda}{2}}\frac{1+e^{-\frac{\lambda}{2}}}{1-e^{-\frac{\lambda}{2}}}C(\omega)\lVert h_1-h_2\rVert_\infty' ,
\end{split}
\]
for $\mathbb P$-a.e. $\omega \in \Omega$ and $x\in X$. Therefore, 
\[
\lVert \mathcal Th_1-\mathcal Th_2\rVert_\infty'\le ce^{\frac{\lambda}{2}}\frac{1+e^{-\frac{\lambda}{2}}}{1-e^{-\frac{\lambda}{2}}}\lVert h_1-h_2\rVert_\infty' ,
\]
 for $h_1, h_2\in Y_\infty$. If $c>0$ is such that 
\[
ce^{\frac{\lambda}{2}}\frac{1+e^{-\frac{\lambda}{2}}}{1-e^{-\frac{\lambda}{2}}}<1,
\]
we conclude that $h\mapsto \mathcal Th$ is a contraction on $(Y_\infty,\|\cdot\|_\infty')$. Hence, $h\mapsto \mathcal Th$ has a unique fixed point $h\in Y_\infty$. Thus, 
\[
h(\omega, x)=\sum_{n\in \Z}\mathcal G(\sigma^{-n} (\omega), n)p(n, \omega, x),
\]
for $\mathbb P$-a.e. $\omega \in \Omega$ and $x\in X$. Hence, we have that 
\[
\begin{split}
&h (\sigma ( \omega), x_1(1, \omega, \xi, \eta)+y(1, \omega, \eta)) \\
&=\sum_{n\in \Z}\mathcal G(\sigma^{-(n-1)}(\omega), n) \bigg (g_{\sigma^{-n}(\omega)}(x_1(-n, \omega, \xi, \eta) 
+y(-n, \omega, \eta) \\
&\phantom{=}
+h(\sigma^{-n}(\omega), x_1(-n, \omega, \xi, \eta)+y(-n, \omega, \eta) )) \bigg )\\
&\phantom{=}-\sum_{n\in \Z}\mathcal G(\sigma^{-(n-1)}(\omega), n)f_{\sigma^{-n}(\omega)}(x_1(-n, \omega, \xi, \eta)+y(-n, \omega, \eta) )\\
&=\Pi^1(\sigma(\omega))g_\omega(\xi+\eta+h(\omega, \xi+ \eta))\\
&\phantom{=}+A(\omega)\sum_{n=0}^\infty \mathcal G(\sigma^{-n}(\omega), n) \bigg (g_{\sigma^{-(n+1)}(\omega)}(x_1(-(n+1), \omega, \xi, \eta) \\
&\phantom{=}+y(-(n+1), \omega, \eta) +h(\sigma^{-(n+1)}(\omega), x_1(-(n+1), \omega, \xi, \eta)+y(-(n+1), \omega, \eta))) \bigg ) \\
&\phantom{=}-\Pi^1(\sigma (\omega))f_\omega(\xi +\eta)\\
&\phantom{=} -A(\omega)\sum_{n=0}^\infty \mathcal G(\sigma^{-n}(\omega), n)f_{\sigma^{-(n+1)}(\omega)}(x_1(-(n+1), \omega, \xi, \eta)+y(-(n+1), \omega, \eta))\\
&\phantom{=}+\Pi^3(\sigma (\omega))g_\omega(\xi+\eta+h(\omega, \xi+ \eta)) \\
&\phantom{=}+A(\omega)\sum_{n=-\infty}^{-1}\mathcal G(\sigma^{-n}(\omega), n) \bigg (g_{\sigma^{-(n+1)}(\omega)}(x_1(-(n+1), \omega, \xi, \eta) \\
&\phantom{=}+y(-(n+1), \omega, \eta) +h(\sigma^{-(n+1)}(\omega), x_1(-(n+1), \omega, \xi, \eta)+y(-(n+1), \omega, \eta))) \bigg )\\
&\phantom{=}-\Pi^3(\sigma (\omega))f_\omega(\xi+\eta) \\
&\phantom{=}-A(\omega)\sum_{n=-\infty}^{-1} \mathcal G(\sigma^{-n}(\omega), n)f_{\sigma^{-(n+1)}(\omega)}(x_1(-(n+1), \omega, \xi, \eta)+y(-(n+1), \omega, \eta))\\
&=A(\omega)h(\omega, x)+g_\omega(\xi+\eta+h(\omega, \xi+ \eta))-f_\omega(\xi+\eta),
\end{split}
\]
and thus
\begin{equation}\label{line}
h (\sigma ( \omega), A(\omega)x+f_\omega(x)) =A(\omega)h(\omega, x)+g_\omega(x+h(\omega, x))-f_\omega(x),
\end{equation}
for $\mathbb P$-a.e. $\omega \in \Omega$ and $x\in X$.

For $\omega \in \Omega$, we define $H_\omega \colon X\to X$ by 
\[
H_\omega(x)=x+h(\omega, x), \quad x\in X.
\]
By~\eqref{line}, we have that~\eqref{1lin} holds. 

In analogous manner, for $\bar h \in Y_\infty$ we define $\mathcal T' \bar h$ by
\[
(\mathcal T' \bar h)(\omega, x)=\sum_{n\in \Z}\mathcal G(\sigma^{-n} (\omega), n)r(n, \omega, x),
\]
where
\[
\begin{split}
r(n, \omega, x) &=f_{\sigma^{-(n+1)}(\omega)}(x_2(-(n+1), \omega, \xi, \eta)+y(-(n+1), \omega, \eta) \\
&\phantom{=}+\bar h(\sigma^{-(n+1)}(\omega), x_2(-(n+1), \omega, \xi, \eta))) \\
&\phantom{=}-g_{\sigma^{-(n+1)}(\omega)}(x_2(-(n+1), \omega, \xi, \eta)),
\end{split}
\]
for $n\in \Z$, $\omega \in \Omega$ and $x\in X$, where $\xi=(\Id-\Pi^2(\omega))x$, $\eta=x-\xi$. By arguing in  a same manner as for $\mathcal T$, we have that (for $c>0$ sufficiently small) $\bar h\mapsto \mathcal T' \bar h$ is a contraction and thus it has a unique fixed point $\bar h \in Y_\infty$. For $\omega \in \Omega$, we define
$\bar H_\omega \colon X\to X$ by
\[
\bar H_\omega(x)=x+\bar h(\omega, x), \quad x\in X.
\]
Then, we have that 
\begin{equation}\label{2lin}
\bar H_{\sigma (\omega)} \circ (A(\omega)+g_\omega)=(A(\omega)+f_\omega)\circ \bar H_\omega.
\end{equation}
In order to complete the  proof  of our theorem we need the following auxiliary result.

\begin{lemma}\label{lem: uniqueness}
Let $(Q_\omega)_{\omega \in \Omega}$ be a family of continuous  maps $Q_\omega \colon X\to X$ such that:
\begin{itemize}
\item for $\mathbb P$-a.e. $\omega \in \Omega$,
\begin{equation}\label{eq: conj lemma}
Q_{\sigma(\omega)}\circ (A(\omega)+f_\omega)=(A(\omega)+g_\omega)\circ Q_\omega;
\end{equation}
\item for $x\in X$, $\omega \mapsto Q_\omega(x)$ is measurable;
\item there exists a tempered random variable $T\colon \Omega \to (0, \infty)$ such that
\begin{equation}\label{eq: bound}
\lVert Q_\omega-\Id\rVert_\infty \le T(\omega), \quad \text{for $\mathbb P$-a.e. $\omega \in \Omega$;}
\end{equation}
\item $Q_{\omega}(x)-x \in E^{s,u}(\omega) $ for $\mathbb P$-a.e. $\omega \in \Omega$ and every $x\in X$.
\end{itemize}
Then, $q(\omega ,x):=Q_{\omega}(x)-x  $ is a fixed point of $\mathcal{T}$. In particular, the family of maps $(Q_\omega)_{\omega \in \Omega}$ with the above properties is essentially unique.
\end{lemma}
\begin{proof}[Proof of the lemma]
Given $\omega \in \Omega$ and $x\in X$, let us consider $\xi=(\Id-\Pi^2(\omega))x$, $\eta=x-\xi$ and $(x_1(n,\omega,\xi,\eta),y(n,\omega , \eta))_{n\in \mathbb{Z}}$ as before. We start observing that condition \eqref{eq: conj lemma} implies that 
\begin{displaymath}
\begin{split}
q(\sigma(\omega ) ,A(\omega)x+f_\omega (x))&= Q_{\sigma(\omega)}(A(\omega)x+f_\omega (x))-\left( A(\omega)x+f_\omega (x)\right)\\
&=\left( A(\omega)+g_\omega \right)( Q_\omega(x))- A(\omega)x-f_\omega (x)\\
&=A(\omega)q(\omega,x)+g_\omega(x+q(\omega,x))-f_\omega(x).
\end{split}
\end{displaymath}
Thus, recalling that $A(\omega)x+f_\omega (x)=x_1(1,\omega,\xi,\eta)+y(1,\omega , \eta)$ we can rewrite the previous equality as
\begin{displaymath}
\begin{split}
&q(\sigma(\omega ) ,x_1(1,\omega,\xi,\eta)+y(1,\omega , \eta)) \\
&=A(\omega)q(\omega,x_1(0,\omega,\xi,\eta)+y(0,\omega , \eta))\\
&+g_\omega(x_1(0,\omega,\xi,\eta)+y(0,\omega , \eta)+q(\omega,x_1(0,\omega,\xi,\eta)+y(0,\omega , \eta)))\\
&-f_\omega(x_1(0,\omega,\xi,\eta)+y(0,\omega , \eta)),
\end{split}
\end{displaymath}
which implies that
\begin{displaymath}
\begin{split}
&q(\omega  ,x) \\
&=A(\sigma^{-1}(\omega ))q(\sigma^{-1}(\omega ),x_1(-1,\omega,\xi,\eta)+y(-1,\omega , \eta))\\
&+g_{\sigma^{-1}(\omega )}(x_1(-1,\omega,\xi,\eta)+y(-1,\omega , \eta)+q(\sigma^{-1}(\omega ),x_1(-1,\omega,\xi,\eta)+y(-1,\omega , \eta)))\\
&-f_{\sigma^{-1}(\omega )}(x_1(-1,\omega,\xi,\eta)+y(-1,\omega , \eta)).
\end{split}
\end{displaymath}
Iterating this formula we conclude that for every $k> 0$,
\begin{displaymath}
\begin{split}
q(\omega  ,x)&=\cA(\sigma^{-k}(\omega),k)q(\sigma^{-k}(\omega ),x_1(-k,\omega,\xi,\eta)+y(-k,\omega , \eta))\\
&+\sum_{n=0}^{k-1} \cA(\sigma^{-n}(\omega),n)p(n,\omega,x)\\
\end{split}
\end{displaymath}
where $p(n,\omega,x)$ is as defined in \eqref{eq: def p} with $h=q$. Thus, using property \eqref{eq: invariance projections} we conclude that  
\begin{displaymath}
\begin{split}
\Pi^1(\omega) q(\omega  ,x)&=\cA(\sigma^{-k}(\omega),k)\Pi^1(\sigma^{-k}(\omega))q(\sigma^{-k}(\omega ),x_1(-k,\omega,\xi,\eta)+y(-k,\omega , \eta))\\
&+\sum_{n=0}^{k-1} \cA(\sigma^{-n}(\omega),n)\Pi^1(\sigma^{-n}(\omega))p(n,\omega,x)\\
&=\mathcal{G}(\sigma^{-k}(\omega),k)q(\sigma^{-k}(\omega ),x_1(-k,\omega,\xi,\eta)+y(-k,\omega , \eta))\\
&+\sum_{n=0}^{k-1} \mathcal{G} (\sigma^{-n}(\omega),n)p(n,\omega,x)\\
\end{split}
\end{displaymath}
for every $k>0$. It follows from~\eqref{eq: bound} that  for $\mathbb P$-a.e. $\omega \in \Omega$ and every $k>0$,
\begin{displaymath}
|q(\sigma^{-k}(\omega ),x_1(-k,\omega,\xi,\eta)+y(-k,\omega , \eta))|\leq T(\sigma^{-k}(\omega )).
\end{displaymath}
Combining this observation with \eqref{eq: def tempered} and \eqref{d1} we conclude that 
\begin{displaymath}
|\mathcal{G}(\sigma^{-k}(\omega),k)q(\sigma^{-k}(\omega ),x_1(-k,\omega,\xi,\eta)+y(-k,\omega , \eta)|\xrightarrow{k\to +\infty}0
\end{displaymath}
for $\mathbb P$-a.e. $\omega \in \Omega$. Consequently,
\begin{displaymath}
\begin{split}
\Pi^1(\omega) q(\omega  ,x)&=\sum_{n=0}^{+\infty} \mathcal{G} (\sigma^{-n}(\omega),n)p(n,\omega,x).\\
\end{split}
\end{displaymath}
Similarly, one can prove that 
\begin{displaymath}
\begin{split}
\Pi^3(\omega) q(\omega  ,x)&=\sum_{n=-\infty}^{-1} \mathcal{G} (\sigma^{-n}(\omega),n)p(n,\omega,x).\\
\end{split}
\end{displaymath}
Therefore, since $q(\omega,x)=Q_{\omega}(x)-x \in E^{s,u}(\omega) $ for $\mathbb P$-a.e. $\omega \in \Omega$ and every $x\in X$, we conclude that
\begin{displaymath}
\begin{split}
q(\omega  ,x)&=\Pi^1(\omega) q(\omega  ,x)+\Pi^3(\omega) q(\omega  ,x)\\
&=\sum_{n\in \Z} \mathcal{G} (\sigma^{-n}(\omega),n)p(n,\omega,x)\\
&=(\mathcal{T}q)(\omega,x),
\end{split}
\end{displaymath}
for $\mathbb P$-a.e. $\omega \in \Omega$ and every $x\in X$. Therefore, $q$ is a fixed point for $\mathcal T$ as claimed. 
 The uniqueness property follows easily from previous observations and from the fact that $\mathcal{T}$ is a contraction. The proof of the lemma is completed.
\end{proof}
Our objective now is to prove that $H_\omega\circ \bar{H}_\omega =\Id$ and $\bar{H}_\omega \circ H_\omega =\Id$. We begin by observing that~\eqref{1lin} and~\eqref{2lin} readily imply that
\begin{displaymath}
\bar{H}_{\sigma (\omega)}\circ H_{\sigma (\omega)}\circ \left(A(\omega)+f_\omega \right)=\left(A(\omega)+f_\omega \right)\circ \bar{H}_{\omega}\circ H_{\omega},
\end{displaymath}
for $\mathbb{P}$-a.e $\omega\in \Omega$. Moreover, note that
\begin{displaymath}
(\bar{H}_{\omega}\circ H_{\omega})(x)-x=(\bar{H}_{\omega}\circ H_{\omega})(x)-H_{\omega}(x)+ H_{\omega}(x)-x,
\end{displaymath}
and thus by using the  properties of $\bar{H}$ and $H$ we conclude that $\bar{H}_{\omega}\circ H_{\omega}(x)-x \in E^{s,u}(\omega) $ for $\mathbb P$-a.e. $\omega \in \Omega$ and every $x\in X$. Moreover,  there exists a tempered random variable $\tilde{T}\colon \Omega \to (0, \infty)$ such that
\begin{equation*}
\lVert \bar{H}_{\omega}\circ H_{\omega}-\Id\rVert_\infty \le \tilde{T}(\omega), \quad \text{for $\mathbb P$-a.e. $\omega \in \Omega$.}
\end{equation*}
Finally, since $H_\omega$ and $\bar{H}_\omega$ are strongly measurable we have that $\omega \to (\bar{H}_{\omega}\circ H_{\omega})(x)$ is measurable for every $x\in X$. Thus, it follows from the uniqueness given in Lemma \ref{lem: uniqueness} (applied to the case when $g_\omega=f_\omega$) that $\bar{H}_{\omega}\circ H_{\omega}=\Id$ for $\mathbb P$-a.e. $\omega \in \Omega$. By changing the roles of $f_\omega$ and $g_\omega$ and $H_{\omega}$ and $\bar{H}_{\omega}$ we also conclude that $H_{\omega}\circ \bar{H}_{\omega}=\Id$, for $\mathbb P$-a.e. $\omega \in \Omega$. In particular, $H_{\omega}$ is a homeomorphism for $\mathbb P$-a.e. $\omega \in \Omega$. The proof of the theorem is completed. 
\end{proof}

\section{H\"older linearization of random dynamical systems}
In this section we show that under some additional assumptions the conjugacies obtained in Theorem \ref{theo: hg} are H\"older continuous. We retain all the notation and the  assumptions introduced in the previous section. Moreover, we assume that 
$A(\omega)$ is invertible for $\mathbb{P}$-a.e $\omega \in \Omega$.
In this case, for $\mathbb{P}$-a.e $\omega \in \Omega$, we consider
\[
\cA(\omega, n)=\begin{cases}
A(\sigma^{n-1}(\omega))\cdots A(\omega) & \text{if $n>0$,}\\
\Id & \text{if $n=0$;}\\
A(\sigma^{-|n|}(\omega))^{-1}\cdots A(\sigma^{-1}(\omega))^{-1} & \text{if $n<0$.}
\end{cases}
\]
Finally, suppose that 
 there exists $\rho >0$ and a tempered random variable $Z\colon \Omega \to [1, +\infty)$ 
 such that
\begin{equation}\label{eq: growth of A}
\|\cA(\omega, n)\|\leq  Z(\omega) e^{\rho |n|}, 
\end{equation}
for  $\mathbb{P}\text{-a.e. }\omega\in \Omega  \text{ and } n\in \mathbb{Z}$.
We may assume without loss of generality that $\rho \geq \lambda$ (where $\lambda>0$ is given in \eqref{d1} and \eqref{d2}) and that 
\begin{equation}\label{KZ}
K(\omega) \le Z(\omega), \quad \text{for $\mathbb P$-a.e. $\omega \in \Omega$.}
\end{equation}
Set $\alpha_0:=\lambda/\rho \in (0, 1]$.
\begin{remark}
Observe that condition \eqref{eq: growth of A} is satisfied, for instance, whenever
\begin{displaymath}
\int \log \|A(\omega)\|d\mathbb{P}<+\infty \text{ and } \int \log \|A(\omega)^{-1}\|d\mathbb{P}<+\infty.
\end{displaymath}
\end{remark}
Observe also that it follows from~\eqref{d1} and~\eqref{d2} that 
\begin{equation}\label{eq: def M}
\|\Pi^2(\omega)\|\leq M(\omega) \text{ for } \mathbb{P}\text{-a.e. } \omega \in \Omega,
\end{equation}
where
\[
M(\omega)=1+2K(\omega), \quad \omega \in \Omega. 
\]

\begin{theorem}\label{HL}
Let $\alpha\in (0,\alpha_0)$. 
There exists a tempered random variable $d:\Omega \to (0,+\infty)$ such that, if
\begin{equation}\label{l1V2}
|f_\omega(x)-f_\omega(z)|\le d(\omega)|x-z|
\end{equation}
and
\begin{equation}\label{l2V2}
|g_\omega(x)-g_\omega(z)|\le d(\omega)|x-z|
\end{equation}
for $\mathbb{P}$-a.e. $\omega \in \Omega$ and every $x,z\in X$, then the conjugacies $H_\omega$ and the associated inverses $H_\omega ^{-1}$ given by Theorem \ref{theo: hg} are $\alpha$-H\"older continuous when restricted to bounded subsets of $X$ for $\mathbb{P}$-a.e $\omega\in \Omega$. More precisely, given a bounded subset $\tilde{X}\subset X$, there exists a constant $T>0$ depending only on $\tilde{X}$ such that
\begin{equation*}
|H_\omega(x)-H_\omega(z)|\leq T|x-z|^\alpha
\end{equation*}
and 
\begin{equation*}
| H_\omega^{-1}(x)- H_\omega^{-1}(z)|\leq T|x-z|^\alpha
\end{equation*}
for $\mathbb{P}$-a.e $\omega\in \Omega$ and every $x,z\in \tilde{X}$.
\end{theorem}

\begin{proof} Let $c\in (0,1)$ be a sufficiently small constant as in \eqref{l1} and \eqref{l2} so that Theorem \ref{theo: hg} holds. We are going to build the desired map $d:\Omega \to (0,+\infty)$ with the property that 
\begin{equation}\label{eq: aux c}
d(\omega)\le \frac{c}{Z(\sigma(\omega))}
\end{equation}
for $\mathbb{P}$-a.e. $\omega \in \Omega$. In particular, due to~\eqref{KZ},  whenever \eqref{l1V2} and \eqref{l2V2} are satisfied the same is true for \eqref{l1} and \eqref{l2}. 

Let $\varepsilon>0$ be such that $\alpha \rho +10\varepsilon -\lambda<0$. It follows from our assumptions and~\cite[Proposition 4.3.3]{Arnold} that there exists a tempered random variable $N\colon \Omega \to [1, \infty)$ such that for $\mathbb P$-a.e. $\omega \in \Omega$ and $n\in \Z$, 
\begin{equation}\label{eq: def of N}
\frac{1}{N(\omega)}\leq Z(\omega)( 2K(\omega)+M(\omega))+ Z(\sigma(\omega))+2D(\omega)\leq N(\omega) 
\end{equation}
and
\begin{equation}\label{eq: def of N2}
N(\omega)e^{-\varepsilon|n|}\leq  N(\sigma^n (\omega))\leq N(\omega)e^{\varepsilon|n|}.
\end{equation}
Thus, using \eqref{eq: f and g bounded}, \eqref{l1V2} and \eqref{eq: def of N},
\begin{equation}\label{eq: f is holder}
\begin{split}
|f_\omega (x)-f_\omega(z)|&=|f_\omega (x)-f_\omega(z)|^{1-\alpha}|f_\omega (x)-f_\omega(z)|^\alpha \\
&\leq \left(2D(\omega)\right)^{1-\alpha}d(\omega)^\alpha|x-z|^\alpha\\
&\leq N(\omega)d(\omega)^\alpha|x-z|^\alpha
\end{split}
\end{equation}
for $\mathbb{P}$-a.e. $\omega \in \Omega$ and every $x,z\in X$. Similarly, using \eqref{eq: f and g bounded}, \eqref{l2V2} and \eqref{eq: def of N},
\begin{equation}\label{eq: g is holder}
\begin{split}
|g_\omega (x)-g_\omega(z)|\leq N(\omega)d(\omega)^\alpha|x-z|^\alpha
\end{split}
\end{equation}
for $\mathbb{P}$-a.e. $\omega \in \Omega$ and every $x,z\in X$.

Let $Y_\infty$ be as in the proof of Theorem~\ref{theo: hg}. Furthermore, let $Y^\alpha\subset Y_\infty$ consist of all $h\in Y_\infty$ with the property that 
\begin{equation}\label{eq: Y holder}
|h(\omega,x)-h(\omega,z)|\leq  |x-z|^\alpha, 
\end{equation}
for $\mathbb{P}$-a.e. $\omega \in \Omega$ and every $x,z \in X$. 
We claim that whenever $c>0$ is sufficiently small, there exists $d:\Omega \to (0,+\infty)$ such that if \eqref{l1V2} and \eqref{l2V2} are satisfied then $\mathcal{T}(Y^\alpha)\subset Y^\alpha$,  where $\mathcal{T}$ is as in the proof of Theorem \ref{theo: hg}. Indeed, given $h\in Y^\alpha$, for $x,z\in X$ and $\mathbb{P}$-a.e $\omega\in \Omega$ we have that 
\begin{equation}\label{eq: first est T}
\begin{split}
|(\mathcal T h)(\omega, x)-(\mathcal T h)(\omega, z)|&=\left|\sum_{n\in \Z}\mathcal G(\sigma^{-n} (\omega), n)\left( p(n, \omega, x)-p(n,\omega ,z)\right)\right|\\
&\leq \sum_{n\in \Z}\| \mathcal G(\sigma^{-n} (\omega), n)\| \left| p(n, \omega, x)-p(n,\omega ,z)\right|\\
&\leq \sum_{n\in \Z}K(\sigma^{-n}(\omega))e^{-\lambda |n|} \left| p(n, \omega, x)-p(n,\omega ,z)\right|\\
&\leq N(\omega)\sum_{n\in \Z}e^{(-\lambda+\varepsilon) |n|} \left| p(n, \omega, x)-p(n,\omega ,z)\right|,
\end{split}
\end{equation}
where $p(n, \omega, x)$ and $p(n,\omega ,z)$ are as in \eqref{eq: def p}.  Observe that we have used that $K(\omega) \le N(\omega)$ for $\mathbb P$-a.e. $\omega \in \Omega$ which follows from~\eqref{eq: def of N}.

Our objective now is to estimate the size of $\left| p(n, \omega, x)-p(n,\omega ,z)\right|$. In order to do it we are going to need several lemmas. Let us consider $\xi_x=(\Id-\Pi^2(\omega))x$, $\eta_x=x-\xi_x$, $\xi_z=(\Id-\Pi^2(\omega))z$ and $\eta_z=z-\xi_z$. Observe that, although all these parameters depend on $\omega$, there will be no confusion in omitting this from the notation itself.

\begin{lemma}
For $\mathbb{P}$-a.e $\omega \in \Omega$ and every $x,z\in X$ and $n\in \mathbb{Z}$,
\begin{equation}\label{eq: est y}
\begin{split}
\left|y(n, \omega, \eta_x)-y(n, \omega, \eta_z)\right|
&\leq N(\omega)e^{\rho|n|}|x-z|.
\end{split}
\end{equation}
\end{lemma}
\begin{proof}[Proof of the lemma]
Since $A(\omega)$ is invertible for $\mathbb{P}$-a.e $\omega \in \Omega$, it follows from \eqref{S1} that $y(n,\omega,\eta)=\cA(\omega,n)\eta$ for $\mathbb{P}$-a.e $\omega \in \Omega$ and every $\eta
\in E^c(\omega)$. Thus, using \eqref{eq: growth of A} and \eqref{eq: def M} we get that
\begin{equation*}
\begin{split}
\left|y(n, \omega, \eta_x)-y(n, \omega, \eta_z)\right|&=\left| \cA(\omega,n)\Pi^2(\omega) (x-z)\right|\\
&\leq M(\omega)Z(\omega) e^{\rho |n|}|x-z|,
\end{split}
\end{equation*}
which together with ~\eqref{eq: def of N} yields the desired result. 
\end{proof}

\begin{lemma}
For $\mathbb{P}$-a.e $\omega \in \Omega$ and every $x,z\in X$ and $n> 0$,
\begin{equation}\label{eq: estimative x future}
\begin{split}
\left|x_1(n, \omega, \xi_x, \eta_x)-x_1(n, \omega, \xi_z, \eta_z)\right| &\leq n N(\omega)\left(e^{\rho}+c\right)^{n}|x-z|.\\
\end{split}
\end{equation}
Similarly, whenever $c>0$ is small enough,
\begin{equation}\label{eq: estimative x past}
\begin{split}
\left|x_1(n, \omega, \xi_x, \eta_x)-x_1(n, \omega, \xi_z, \eta_z)\right| & \leq  |n|N(\omega) \left(\frac{e^{\rho }}{1-ce^\rho}\right)^{|n|+1}|x-z|
\end{split}
\end{equation}
for $\mathbb{P}$-a.e $\omega \in \Omega$ and every $x,z\in X$ and $n<0$.
\end{lemma}

\begin{proof}[Proof of the lemma]
We first introduce some auxiliary notation. For $\mathbb P$-a.e. $\omega \in \Omega$ and $x\in X$, set
\[
|x|_\omega :=\sup_{n\in \Z} (|\cA(\omega, n)x| e^{-\rho |n|} ).
\]
It is easy to verify  that 
\begin{equation}\label{3l}
|x| \le |x|_\omega \le Z(\omega) |x|, 
\end{equation}
and
\begin{equation}\label{4l}
|\cA(\omega, n) x|_{\sigma^n (\omega)} \le e^{\rho | n|} |x|_\omega, 
\end{equation}
for $\mathbb P$-a.e. $\omega \in \Omega$, $n\in \Z$ and $x\in X$. We observe that it follows from~\eqref{l1V2}, \eqref{eq: aux c} and~\eqref{3l} that 
\begin{equation}\label{6l}
|f_\omega(x)-f_\omega(y)|_{\sigma (\omega)} \le c|x-y|, \quad \text{for $\mathbb P$-a.e. $\omega \in \Omega$ and $x, y\in X$.}
\end{equation}
We will now establish~\eqref{eq: estimative x future}. Using~\eqref{eq: growth of A}, \eqref{4l} and~\eqref{6l}, we obtain that for $\mathbb P$-a.e. $\omega \in \Omega$ and $n\ge 0$, 
\[
\begin{split}
&\left|x_1(n+1, \omega, \xi_x, \eta_x)-x_1(n+1, \omega, \xi_z, \eta_z)\right|_{\sigma^{n+1} (\omega)}\\
& \leq \left|A(\sigma^n(\omega))\left( x_1(n, \omega, \xi_x, \eta_x)-x_1(n, \omega, \xi_z, \eta_z)\right) \right|_{\sigma^{n+1} (\omega)}\\
&+\left|f_{\sigma^n(\omega)}(x_1(n, \omega, \xi_x, \eta_x)+y(n, \omega,\eta_x))-f_{\sigma^n(\omega)}(x_1(n, \omega, \xi_z, \eta_z)+y(n, \omega,\eta_z)) \right|_{\sigma^{n+1} (\omega)}\\
&\leq e^\rho \left |x_1(n, \omega, \xi_x, \eta_x)-x_1(n, \omega, \xi_z, \eta_z) \right |_{\sigma^n (\omega)} \\
&\phantom{\leq}+c\left |x_1(n, \omega, \xi_x, \eta_x)+y(n, \omega,\eta_x))-x_1(n, \omega, \xi_z, \eta_z)-y(n, \omega,\eta_z) \right | \\
&\leq e^\rho \left |x_1(n, \omega, \xi_x, \eta_x)-x_1(n, \omega, \xi_z, \eta_z) \right |_{\sigma^n (\omega)} \\
&\phantom{\leq}+c\left |x_1(n, \omega, \xi_x, \eta_x)-x_1(n, \omega, \xi_z, \eta_z)\right |_{\sigma^n (\omega)}+c \left | y(n, \omega,\eta_x))-y(n, \omega,\eta_z)  \right | \\
&= (e^\rho+c) \left| x_1(n, \omega, \xi_x, \eta_x)-x_1(n, \omega, \xi_z, \eta_z) \right|_{\sigma^n (\omega)}+c|\cA(\omega,n)\eta_x-\cA(\omega,n)\eta_z|\\
&\leq (e^\rho+c) \left| x_1(n, \omega, \xi_x, \eta_x)-x_1(n, \omega, \xi_z, \eta_z) \right|_{\sigma^n (\omega)}+Z(\omega) ce^{\rho n}|\eta_x-\eta_z|.
\end{split}
\]
Proceeding recursively and using~\eqref{3l}, we conclude that 
\begin{displaymath}
\begin{split}
&\left|x_1(n+1, \omega, \xi_x, \eta_x)-x_1(n+1, \omega, \xi_z, \eta_z)\right| \\
&\leq \left|x_1(n+1, \omega, \xi_x, \eta_x)-x_1(n+1, \omega, \xi_z, \eta_z)\right|_{\sigma^{n+1} (\omega)}\\
&\leq (e^\rho+c)^{n+1} \left| x_1(0, \omega, \xi_x, \eta_x)-x_1(0, \omega, \xi_z, \eta_z) \right|_\omega\\
&\phantom{\leq} +cZ(\omega)|\eta_x-\eta_z|\sum_{j=0}^n(e^\rho+c)^j e^{\rho (n-j)} \\
&\leq (e^\rho+c)^{n+1} \left| \xi_x-\xi_z \right|_\omega+cZ(\omega)|\eta_x-\eta_z|\sum_{j=0}^n(e^\rho+c)^n\\
&\leq (n+1)(e^\rho+c)^{n+1}Z(\omega) \left(\left| \xi_x-\xi_z \right|+|\eta_x-\eta_z|\right)\\
&\leq (n+1)(e^\rho+c)^{n+1} Z(\omega) \left(2K(\omega)+M(\omega)\right)\left|x-z\right|,
\end{split}
\end{displaymath}
which implies \eqref{eq: estimative x future}.

We now prove \eqref{eq: estimative x past}. Given $\omega\in \Omega$ and $\eta \in E^c(\omega)$ let us consider $F_{\omega}^\eta: E^{s,u}(\omega)\to E^{s,u}(\sigma(\omega))$ defined by
\begin{equation}\label{eq: def F}
F_{\omega}^\eta(\xi)=A(\omega)\xi+f_\omega (\xi+\eta).
\end{equation}
It follows from our hypothesis \eqref{eq: A+f invert} that $F_{\omega}^\eta$ is invertible for $\mathbb{P}$-a.e. $\omega \in \Omega$. Moreover, one can easily check that its inverse is given by 
\begin{equation}\label{eq: inverse of F}
\left( F_{\omega}^\eta\right)^{-1}(\xi)=A(\omega)^{-1}\xi-A(\omega)^{-1} \left(f_\omega\left(\left( F_{\omega}^\eta\right)^{-1}(\xi)+\eta\right)\right).
\end{equation}
Consequently, for $\mathbb P$-a.e. $\omega \in \Omega$ and every $\eta, \theta\in E^c(\omega)$ and $\xi ,\zeta \in E^{s,u}(\sigma(\omega))$,  we have that 
\[
\begin{split}
&\left| \left( F_{\omega}^{\eta}\right)^{-1}(\xi)-\left( F_{\omega}^{\theta}\right)^{-1}(\zeta)\right|_{\omega  } \\
&\leq \left| A(\omega)^{-1}\left(\xi-\zeta\right) \right|_{\omega } \\
&\phantom{\leq}+\left| A(\omega)^{-1}\left( f_\omega \left( \left( F_{\omega}^\eta\right)^{-1}(\xi)+\eta\right)-f_\omega\left( \left( F_{\omega}^\theta\right)^{-1}(\zeta)+\theta\right) \right)  \right|_{\omega}\\
&\leq e^\rho |\xi-\zeta|_{\sigma(\omega )}+e^\rho \left| f_\omega \left( \left( F_{\omega}^\eta\right)^{-1}(\xi)+\eta\right)-f_\omega\left( \left( F_{\omega}^\theta\right)^{-1}(\zeta)+\theta\right) \right|_{\sigma(\omega )} \\
&\leq e^\rho |\xi-\zeta|_{\sigma(\omega )} +ce^\rho  \left( \left| \left( F_{\omega}^\eta\right)^{-1}(\xi)-\left( F_{\omega}^\theta\right)^{-1}(\zeta) \right|+\left|\eta -\theta\right| \right)\\
&\leq e^\rho |\xi-\zeta|_{\sigma(\omega )}+ce^\rho  \left( \left| \left( F_{\omega}^\eta\right)^{-1}(\xi)-\left( F_{\omega}^\theta\right)^{-1}(\zeta) \right|_{\omega}+\left|\eta -\theta\right| \right)\\
&\leq e^\rho \left(|\xi-\zeta|_{\sigma(\omega )}+\left|\eta -\theta\right| \right)+ce^\rho\left| \left( F_{\omega}^\eta\right)^{-1}(\xi)-\left( F_{\omega}^\theta\right)^{-1}(\zeta) \right|_{\omega }.
\end{split}
\]
Thus, for $c>0$ sufficiently small we get that
\begin{equation}\label{eq: estimative inverse}
\begin{split}
&\left| \left( F_{\omega}^{\eta}\right)^{-1}(\xi)-\left( F_{\omega}^{\theta}\right)^{-1}(\zeta)\right|_{\omega } \leq \frac{e^\rho}{1-ce^\rho}  \left(|\xi-\zeta|_{\sigma(\omega )}+\left|\eta -\theta\right| \right).
\end{split}
\end{equation}

Now, recalling \eqref{S1} and \eqref{eq: def F} it follows that for every $n>0$
\begin{equation*}
x_1(-n,\omega,\xi_x,\eta_x)=\left( F_{\sigma^{-n}(\omega)}^{\cA(\omega,-n)\eta_x}\right)^{-1}\left( x_1(-(n-1),\omega,\xi_x,\eta_x)\right).
\end{equation*}
By applying \eqref{eq: estimative inverse},  we obtain that
\begin{equation*}
\begin{split}
&\left| x_1(-n,\omega,\xi_x,\eta_x)-x_1(-n,\omega,\xi_z,\eta_z)\right|_{\sigma^{-n}(\omega)} \\
&\leq \left(\frac{e^\rho}{1-ce^\rho} \right)\left|x_1(-(n-1),\omega,\xi_x,\eta_x)-x_1(-(n-1),\omega,\xi_z,\eta_z)\right|_{\sigma^{-(n-1)}(\omega)} \\
&+ \left(\frac{e^\rho}{1-ce^\rho} \right)\left|\cA(\omega,-n)\eta_x-\cA(\omega,-n)\eta_z \right|\\
&\leq \left(\frac{e^\rho}{1-ce^\rho} \right)\left|x_1(-(n-1),\omega,\xi_x,\eta_x)-x_1(-(n-1),\omega,\xi_z,\eta_z)\right|_{\sigma^{-(n-1)}(\omega)} \\
&+Z(\omega) \left(\frac{e^\rho}{1-ce^\rho} \right)e^{\rho n}\left|\eta_x-\eta_z \right|.\\
\end{split}
\end{equation*}
By proceeding recursively,  we conclude that
\begin{equation*}
\begin{split}
&\left| x_1(-n,\omega,\xi_x,\eta_x)-x_1(-n,\omega,\xi_z,\eta_z)\right|\\
&\leq \left| x_1(-n,\omega,\xi_x,\eta_x)-x_1(-n,\omega,\xi_z,\eta_z)\right|_{\sigma^{-n}(\omega)}\\
&\leq \left(\frac{e^\rho}{1-ce^\rho} \right)^{n}\left|x_1(0,\omega,\xi_x,\eta_x)-x_1(0,\omega,\xi_z,\eta_z)\right|_\omega \\
&\phantom{\leq} +Z(\omega) \left| \eta_x-\eta_z \right| \sum_{j=0}^{n-1}\left(\frac{e^\rho}{1-ce^\rho} \right)^{j+1}e^{\rho(n-j)}\\
&\leq \left(\frac{e^\rho}{1-ce^\rho} \right)^{n}\left|\xi_x-\xi_z\right|_\omega + Z(\omega) \left| \eta_x-\eta_z \right| \sum_{j=0}^{n-1}\left(\frac{e^\rho}{1-ce^\rho} \right)^{n+1}\\
&\leq n \left(\frac{e^\rho}{1-ce^\rho} \right)^{n+1} Z(\omega) \left(\left|\xi_x-\xi_z\right| + \left| \eta_x-\eta_z \right| \right)\\
&\leq n Z(\omega) \left(2K(\omega)+M(\omega)\right)\left(\frac{e^\rho}{1-ce^\rho} \right)^{n+1}\left|x-z\right|\\
&\leq n N(\omega)\left(\frac{e^\rho}{1-ce^\rho} \right)^{n+1}\left|x-z\right|,\\
\end{split}
\end{equation*}
which proves our claim \eqref{eq: estimative x past}. The proof of the lemma is completed. 

\end{proof}

In order to simplify notation, let us consider
\begin{displaymath}
\tilde{x}_1(n, \omega, x):=x_1(n, \omega, \xi_x, \eta_x)+y(n, \omega, \eta_x)
\end{displaymath}
and
\begin{displaymath}
\tilde{x}_1(n, \omega, z):=x_1(n, \omega, \xi_z, \eta_z)+y(n, \omega, \eta_z).
\end{displaymath}
Thus, combining \eqref{eq: est y}, \eqref{eq: estimative x future} and \eqref{eq: estimative x past} we get that
\begin{equation}\label{eq: est x tilde}
|\tilde{x}_1(n, \omega, x)-\tilde{x}_1(n, \omega, z)|\leq 2(|n|+1)N(\omega) \left(\frac{e^{\rho }+c}{1-ce^\rho}\right)^{|n|+1}|x-z|
\end{equation}
for $\mathbb{P}$-a.e. $\omega \in \Omega$ and every $x,z\in X$ and $n\in \mathbb{Z}$. Then, using \eqref{eq: def of N2} and \eqref{eq: f is holder} we conclude that
\begin{equation}\label{eq: estimative f}
\begin{split}
&\left|f_{\sigma^{-(n+1)}(\omega)}(\tilde{x}_1(-(n+1), \omega, x)) -f_{\sigma^{-(n+1)}(\omega)}(\tilde{x}_1(-(n+1), \omega, z))\right|\\
&\leq N(\sigma^{-(n+1)}(\omega))d(\sigma^{-(n+1)}(\omega))^\alpha \left|\tilde{x}_1(-(n+1), \omega, x) -\tilde{x}_1(-(n+1), \omega, z)\right|^\alpha\\
&\leq 2(|n|+2)N(\omega) N(\sigma^{-(n+1)}(\omega))d(\sigma^{-(n+1)}(\omega))^\alpha  \left(\frac{e^{\rho }+c}{1-ce^\rho}\right)^{(|n|+2)\alpha}|x-z|^\alpha\\
&\leq 2(|n|+2)N(\omega)^2 d(\sigma^{-(n+1)}(\omega))^\alpha  e^{\varepsilon(|n|+1)}\left(\frac{e^{\rho }+c}{1-ce^\rho}\right)^{(|n|+2)\alpha}|x-z|^\alpha,
\end{split}
\end{equation}
for $\mathbb{P}$-a.e. $\omega \in \Omega$ and every $x,z\in X$ and $n\in \mathbb{Z}$.

Similarly, assuming without loss of generality that $d(\omega)<1$ and using \eqref{l2V2}, \eqref{eq: def of N2}, \eqref{eq: g is holder} and \eqref{eq: Y holder} we get that, for $\mathbb{P}$-a.e. $\omega \in \Omega$ and every $u,v\in X$,
\[
\begin{split}
&\left|g_{\sigma^{-(n+1)}(\omega)}\left(u +h(\sigma^{-(n+1)}(\omega), u)\right) -g_{\sigma^{-(n+1)}(\omega)}\left(v +h(\sigma^{-(n+1)}(\omega), v)\right) \right|\\
&\leq N(\sigma^{-(n+1)}(\omega))d(\sigma^{-(n+1)}(\omega))^\alpha \min  \Big\{ \left|u-v\right| +\left| h(\sigma^{-(n+1)}(\omega), u)- h(\sigma^{-(n+1)}(\omega), v)\right|,\\
&\phantom{=}\Big[\left|u-v\right| +\left| h(\sigma^{-(n+1)}(\omega), u)- h(\sigma^{-(n+1)}(\omega), v)\right|\Big]^\alpha \Big\}\\
& \leq   N(\sigma^{-(n+1)}(\omega))d(\sigma^{-(n+1)}(\omega))^\alpha\min  \Big\{ \left|u-v\right| + \left| u-v\right|^\alpha, \Big[\left|u-v\right|  +\left|  u-v\right|^\alpha\Big]^\alpha \Big\}\\
&\leq N(\sigma^{-(n+1)}(\omega))d(\sigma^{-(n+1)}(\omega))^\alpha \begin{cases}
2^\alpha\left|  u-v\right|^\alpha & \text{if $\left|  u-v\right|>1$  (using the right one)} \\
2\left|  u-v\right|^\alpha & \text{if $\left| u-v\right|\le 1$ (using the left one)}
\end{cases}\\
&\leq 2N(\sigma^{-(n+1)}(\omega))d(\sigma^{-(n+1)}(\omega))^\alpha|u-v|^\alpha\\
&\leq 2N(\omega)d(\sigma^{-(n+1)}(\omega))^\alpha e^{\varepsilon(|n|+1)}|u-v|^\alpha.
\end{split}
\]
Taking $u=\tilde{x}_1(-(n+1), \omega, x)$ and $v=\tilde{x}_1(-(n+1), \omega, z)$ and applying \eqref{eq: est x tilde} we conclude that
\begin{equation}\label{eq: estimative g}
\begin{split}
&\left|g_{\sigma^{-(n+1)}(\omega)}\left(\tilde{x}_1(-(n+1), \omega, x) +h(\sigma^{-(n+1)}(\omega), \tilde{x}_1(-(n+1), \omega, x))\right)\right. \\
&\left. -g_{\sigma^{-(n+1)}(\omega)}\left(\tilde{x}_1(-(n+1), \omega, z) +h(\sigma^{-(n+1)}(\omega), \tilde{x}_1(-(n+1), \omega, z))\right) \right| \\
&\le 4(|n|+2)N(\omega)^2d(\sigma^{-(n+1)} (\omega))^\alpha  e^{\varepsilon(|n|+1)}\left(\frac{e^{\rho }+c}{1-ce^\rho}\right)^{(|n|+2)\alpha}|x-z|^\alpha, 
\end{split}
\end{equation}
for $\mathbb{P}$-a.e. $\omega \in \Omega$ and every $x,z\in X$ and $n\in \mathbb{Z}$.

Thus, recalling the definition of $p(n, \omega, x)$ and $p(n,\omega ,z)$ in \eqref{eq: def p} and using \eqref{eq: estimative f} and \eqref{eq: estimative g},  it follows that  
\begin{equation}\label{eq: estimative p}
\begin{split}
&\left|p(n, \omega, x)-p(n,\omega ,z)\right| \\
&\leq  6(|n|+2)N(\omega)^2d(\sigma^{-(n+1)} (\omega))^\alpha  e^{\varepsilon(|n|+1)}\left(\frac{e^{\rho }+c}{1-ce^\rho}\right)^{(|n|+2)\alpha}|x-z|^\alpha, 
\end{split}
\end{equation}
for $\mathbb{P}$-a.e. $\omega \in \Omega$ and every $x,z\in X$ and $n\in \mathbb{Z}$. Consequently, plugging \eqref{eq: estimative p} into \eqref{eq: first est T}  yields 
\begin{equation}\label{eq: second est T}
\begin{split}
|(\mathcal T h)(\omega, x)-(\mathcal T h)(\omega, z)| 
&\leq 6 N(\omega)^3B'|x-z|^\alpha
\end{split}
\end{equation}
for $\mathbb{P}$-a.e. $\omega \in \Omega$ and every $x,z\in X$ and $n\in \mathbb{Z}$, where
\[
B':=\sum_{n\in \Z}\left((|n|+2)e^{-\lambda |n|} e^{2\varepsilon(|n|+1)} \left(\frac{e^{\rho }+c}{1-ce^\rho}\right)^{(|n|+2)\alpha}d(\sigma^{-(n+1)} (\omega))^\alpha \right).
\]
 Now, using the fact that $\alpha\rho +10\varepsilon -\lambda <0$ we get that if $c$ is sufficiently small then (mind the term $5\varepsilon$)
\begin{displaymath}
B:=\sum_{n\in \Z}\left((|n|+2)e^{-\lambda |n|} e^{5\varepsilon(|n|+1)} \left(\frac{e^{\rho }+c}{1-ce^\rho}\right)^{(|n|+2)\alpha} \right) <+\infty.
\end{displaymath}
Set
\begin{equation}\label{eq: def of c}
d(\omega)=\left(\frac{c}{12BN(\omega)^3}\right)^{1/\alpha}, \quad \omega \in \Omega. 
\end{equation}
It is easy to see (using~\eqref{eq: def of N}) that \eqref{eq: aux c} is satisfied. Furthermore, since $N$ is tempered, we have that $d$ is also tempered.  Moreover, using the first inequality in \eqref{eq: def of N2} it follows that
\begin{equation*}
\begin{split}
d(\sigma^{-(n+1)(\omega)})^\alpha&=\frac{c}{12BN(\sigma^{-(n+1)}(\omega))^3}\\
&\leq \frac{ce^{3\varepsilon (|n|+1)}}{12BN(\omega)^3}\\
&=d(\omega)^\alpha e^{3\varepsilon (|n|+1)}.
\end{split}
\end{equation*}
Inserting this into \eqref{eq: second est T}, we find that 
\begin{equation*}
\begin{split}
&|(\mathcal T h)(\omega, x)-(\mathcal T h)(\omega, z)|\\
&\leq 6 N(\omega)^3d(\omega)^\alpha\sum_{n\in \Z}\left((|n|+2)e^{-\lambda |n|} e^{5\varepsilon(|n|+1)} \left(\frac{e^{\rho }+c}{1-ce^\rho}\right)^{(|n|+2)\alpha} \right)|x-z|^\alpha\\
&=6 N(\omega)^3d(\omega)^\alpha B|x-z|^\alpha\\
&= \frac{c}{2}|x-z|^\alpha\leq |x-z|^\alpha,
\end{split}
\end{equation*}
for $\mathbb{P}$-a.e $\omega \in \Omega$ and every $x,z\in X$. Hence,  $\mathcal{T}(Y^\alpha)\subset Y^\alpha$ as claimed. Thus, since $Y^\alpha$ is a closed subspace of $(Y_\infty, \|\cdot\|'_\infty)$ and $\mathcal{T}:Y_\infty\to Y_\infty$ is a contraction, it follows that the unique fixed point $h$ of $\mathcal{T}$ is in $Y^\alpha$. Therefore, since the conjugacy given by Theorem \ref{theo: hg} is of the form $H_\omega(x)=x+h(\omega,x)$, it follows that 
\begin{displaymath}
\begin{split}
|H_\omega(x)-H_\omega(z)|&=|x+h(\omega,x)-z-h(\omega,z)|\\
&\leq \left(|x-z|^{1-\alpha}+1\right)|x-z|^\alpha
\end{split}
\end{displaymath}
for $\mathbb{P}$-a.e $\omega \in \Omega$ and every $x,z\in X$. Now, given a bounded subset $\tilde{X}\subset X$ there exists a constant $\tilde{T}>1$ depending only on $\tilde{X}$ such that $|x-z|\le \tilde{T}$ for every $x,z\in \tilde{X}$. Consequently, taking $T:=\tilde{T}+1$ we get that
\begin{displaymath}
\begin{split}
|H_\omega(x)-H_\omega(z)|\leq T|x-z|^\alpha
\end{split}
\end{displaymath}
for $\mathbb{P}$-a.e $\omega \in \Omega$ and every $x,z\in \tilde{X}$.  The fact that $H_\omega^{-1}$ satisfies a similar property follows by changing the roles of $f_\omega$ and $g_\omega$ in the previous argument. The proof of the theorem is complete.
\end{proof}

\begin{remark}\label{rem: non ergodic}
Theorems \ref{theo: hg} and \ref{HL} were stated and proved in the case when the base dynamics $(\Omega, \sigma ,\mathbb{P})$ is ergodic. Nevertheless, similar versions with generalized hypothesis can be obtained in the non-ergodic setting simply by combining our results with the Ergodic Decomposition Theorem: if $\mathbb{P}$ is non-ergodic then we can decompose it into ergodic components, each of which ``living" in disjoint subsets of $\Omega$, and apply our results to each of these components. Then, by ``gluing" together each of the respective maps we get the desired result. The main differences in the statements of the results in the non-ergodic case with respect to the ergodic one appear in the constant $c$ in Theorem \ref{theo: hg} and in the H\"older exponent in Theorem \ref{HL}: both will become $\sigma$-invariant maps $c,\alpha:\Omega \to (0,1)$ since these quantities will depend on the ergodic component of $\mathbb{P}$. For instance, instead of getting a constant H\"older exponent that holds for $\mathbb{P}$-a.e $\omega \in \Omega$, we will get a $\sigma$-invariant map $\alpha:\Omega \to (0,1)$ such that each $H_\omega$ and $H_\omega^{-1}$ is $\alpha(\omega)$-H\"older continuous for $\mathbb{P}$-a.e $\omega \in \Omega$.
As already mentioned, in the non-ergodic setting our hypothesis can be relaxed accordingly. For instance, in conditions \eqref{d1}, \eqref{d2} and \eqref{eq: growth of A} the constants $\lambda$ and $\rho$ can be replaced by $\sigma$-invariant maps $\lambda,\rho:\Omega \to (0,+\infty)$. 
\end{remark}

\section{Acknowledgements}
We would like to thank the referee for useful comments that helped us improve our paper.
 L.B. was partially supported by a CNPq-Brazil PQ fellowship under Grant No. 306484/2018-8. D.D. was supported in part by Croatian Science Foundation under the project
IP-2019-04-1239 and by the University of Rijeka under the projects uniri-prirod-18-9 and uniri-pr-prirod-19-16.

\end{document}